\documentclass[11pt, oneside]{amsart}

\usepackage{stmaryrd}

\usepackage[T1]{fontenc}
\usepackage{enumerate}
\usepackage{textcomp}
\usepackage{amsmath}
\usepackage{amsfonts}
\usepackage{amssymb}
\usepackage{amsthm}
\usepackage{tikz}
\usetikzlibrary{cd}
\usepackage{cancel}
\usepackage{multicol}
\usepackage{geometry} 
\usepackage{parskip}
\usepackage{multicol}
\usepackage{hyperref}
\usepackage{xcolor}
\usepackage[polish,english]{babel}
\usepackage{csquotes}
\usepackage{mathrsfs}
\usepackage{mathtools}

\usepackage{bbm}

\newcommand{\cc}[1]{\mathcal{#1}}

\newcommand \bw {\beta\omega}

\newtheorem{theorem}{Theorem}[section]
\newtheorem{problem}[theorem]{Problem}
\newtheorem{lemma}[theorem]{Lemma}
\newtheorem{prop}[theorem]{Proposition}
\newtheorem{corr}[theorem]{Corollary}

\theoremstyle{definition}

\newtheorem{example}[theorem]{Example}

\newtheorem{remark}[theorem]{Remark}

\theoremstyle{definition}
\newtheorem{definition}[theorem]{Definition}

\title[On points avoiding measures]{On points avoiding measures}

\author{Piotr Borodulin--Nadzieja}
\address[Piotr Borodulin-Nadzieja]{Mathematical Institute, University of Wroc\l aw, \ \ \  pl. Grunwaldzki 2, 50-384 Wroc\l aw, Poland}
\email{pborod@math.uni.wroc.pl}

\author{Artsiom Ranchynski}
\address[Artsiom Ranchynski]{Mathematical Institute, University of Wroc\l aw, \ \ \  pl. Grunwaldzki 2, 50-384 Wroc\l aw, Poland}
\email{artsiom.ranchynski@mimuw.edu.pl}

\thanks{The
first author was supported by the National Science Centre, Poland under the Weave-UNISONO call in the
Weave programme 2021/03/Y/ST1/00124. }


\begin{document}

\begin{abstract} We say that an element $x$ of a topological space $X$ avoids measures if for every Borel measure $\mu$ on $X$ if $\mu(\{x\})=0$, then there is an open $U\ni x$ such that $\mu(U)=0$. The negation of this property can viewed as a local version of the property
	of supporting a strictly positive measure. We study points avoiding measures in the general setting as well
	as in the context of $\omega^\ast$, the remainder of Stone-\v{C}ech compactification of $\omega$.
\end{abstract}

\maketitle

\section{Introduction}

We say that a topological space $X$ is homogenuous if for every $x, y \in X$, there is an autohomeomorphism $h\colon X\to X$ such that $h(x)=y$.
For many years, one of the most important problems in set theoretic topology was to find a concrete reason why $\omega^\ast = \beta\omega\setminus\omega$, the space of non-principal ultrafilters on the set of positive integers $\omega$, is not
homogenuous. Zden\'{e}k Frolik showed \cite{Frolik} that $\omega^\ast$ is not homogenuous, but he has not indicated any points witnessing the lack of homogeneity. So, the problem of finding a 'honest proof' (using the
words
of Eric van Douwen, see \cite{Dow-Mill}) of non-homogeneity was still open.

In 1956 Walter Rudin defined the notion of P-point, see \cite{Rudin}. In the topological setting, a point $x\in X$ is a P-point if every $G_\delta$ containing $x$ has nonempty interior containing $x$. It is easy to construct P-points in $\omega^\ast$ under the assumption of Continuum Hypothesis and
it is also easy to see that there are ultrafilters which are not P-points. So, under CH the problem of witnesses for non-homogeneity of $\omega^\ast$ is settled.

One can construct P-points in $\omega^\ast$ under several axioms and it seems that for some time mathematicians believed that P-points do exist in ZFC. Surprisingly,  Saharon Shelah constructed a model in which $\omega^\ast$ does not
have any P-points, see \cite{Shelah}. So, Rudin's approach does not work without additional axioms.

Kenneth Kunen noticed in \cite{Kunen_some}, that the property of being a P-point can be relaxed to a slightly weaker one: we say that a point $x\in X$ is a weak P-point if $x \notin \overline{C}$ for any countable $C\subseteq X$ which does not
contain $x$. Clearly, $\omega^\ast$
contain many points which are not weak P-points (just take a countable subset of $\omega^\ast$ and consider its closure). Actually, it is still not easy to show that there are weak P-points in ZFC, directly. But, Kunen defined a stronger notion
of
$\mathfrak{c}$-OK-points, and showed that that
$\mathfrak{c}$-OK-points do exist without any additional axioms, \cite{Kunen_weak_p}. So, the problem of finding witnesses of non-homogeneity was finally solved.

In this article, we define a seemingly new topological property which may be interesting in this context. We say that a point $x\in X$ avoids measures, if $x$ is not in the support of any measure $\mu$ such that $\mu(\{x\})=0$. Recall that a support
of a probability (Radon) measure $\mu$ is a smallest
closed set $K$ such that $\mu(K)=1$.
Notice that this property is weaker than the property of being a P-point: if $x$ is a P-point and $\mu(\{x\})=0$, then we may find a sequence of open neighbourhoods $(U_n)$ of $x$ such that $\mu(U_n)<1/n$. But then there is an open set $U \ni x$ such
that $U \subseteq U_n$ for each $n$. So, $\mu(U)=0$ and thus $x$ is not in the support of $\mu$.

We are interested in this property because of two
reasons.

First, we may consider its natural modifications: property of avoiding non-atomic measures and property of avoiding purely atomic measures. It turns out that the property of being a  weak P-point can be stated in the language of measures: a point is
a weak P-point if and only if it avoids purely atomic measures. It is
natural to ask what is the connection between weak P-points and points avoiding non-atomic measures. In Section \ref{general-spaces} we present some results along these lines in the general context. I.e., we show that these properties are orthogonal althought the appropriate witnesses are not always cheap. We show that the Stone space of the measure algebra contains a weak P-point and that certain separable space constructed
by Simon contains points avoiding non-atomic measures. Then, in Section \ref{omegaast}, using these results we present examples of weak P-points which does not avoid measures and vice versa.

The second reason why we study ultrafilters avoiding measures goes back to the question of witnesses for non-homogeneity of $\omega^\ast$.  Kunen's construction of $\mathfrak{c}$-OK-point is quite involved. There is no easier ZFC construction producing
weak P-points known so far. But perhaps there is an easier construction of an ultrafilter avoiding non-atomic measures. Note that we can always define a non-atomic measure on $\omega^\ast$. One can define such a measure explicitly, extending the
asymptotic density along an ultrafilter, see e.g. \cite{Van_Douwen}. Or we
can use a theorem of Rudin saying that a compact space does not admit a non-atomic measure if and only if it is scattered (see \cite{Rudin2} and \cite{Pelcz_Semad}). So, if we are able to show in ZFC that there is an
ultrafilter avoiding non-atomic measures, then we would have another witness for non-homogeneity of $\omega^\ast$. 

While trying to construct a point avoiding non-atomic measures, we came up with a new notion: ultrafilter diagonalizing trees along ideals on $2^\omega$. Using this property we were able to construct ultrafilters avoiding measures, which are not
P-points (at least under $CH$). Then we have realized that this notion coincides with Baumgartner's measure zero ultrafilters. This gives us several examples of ultrafilters avoiding non-atomic measures (using results of Baumgartner and Brendle) but
also shows that this approach cannot work in ZFC (as Shelah proved that consistently there are non measure zero ultrafilters).

In Section \ref{q-points} we show that the property of being a Q-point does not imply avoiding non-atomic measures. The last section is devoted to open problems.

\section{Acknowledgements}

We are grateful to Jan van Mill for reading a preliminary version of this article and for sharing with us the proof of Theorem \ref{Janek}.

\section{Preliminaries}

By a \emph{measure} on a Boolean algebra $\cc{B}$ we will mean a finitely additive probability measure (i.e. such that $\mu(1_{\cc{B}})=1$).
By a measure on a topological space $X$ we will mean a $\sigma$-additive Radon probability measure on $Bor(X)$. In particular, we assume that measures on topological spaces are inner regular with respect to compact sets and outer regular with respect
to open sets.

\begin{prop}
	Every measure $\mu$ on a Boolean algebra $\cc{P}(\omega)/fin$ can be uniquely extended to a ($\sigma$-additive, Radon) measure on the Stone space $St(\cc{P}(\omega)/fin)$.
\end{prop}

Hence, there is a bijective correspondence between finitely-additive measures on $\cc{P}(\omega)/fin$ and ($\sigma$-additive, Radon) measures on ${\omega}^{\ast}$.

A measure $\mu$ on a space $X$ is non-atomic if $\mu(\{x\})=0$ for every $x\in X$.

A measure $\mu$ on a Boolean algebra $\cc{B}$ is \emph{non-atomic} (sometimes called also \emph{strongly continuous}) if for every $\varepsilon>0$ there exists a finite partition $\{P_0, \dots, P_n\}$ of unity  with $\mu(P_i)< \varepsilon$ for each
$i\leq n$. Note that a measure $\mu$ on a Boolean algebra $\cc{B}$ is non-atomic if and only if its extension to $St(\mathbb{B})$ is non-atomic.

A measure  on a topological space $X$ is \emph{purely atomic} if $\mu$ is a countable weighted sum of 0-1 measures.

A support of a measure $\mu$ on a space $X$ is $supp(\mu)=\bigcap\{C\subseteq X\colon \text{C is closed and }\mu(C)=1\}$. If the measure is inner regular with respect to compact sets, then $\mu(supp(\mu))=1$, see \cite[411N (d)]{Fremlin}.

\begin{definition}
	A point $x\in X$ \emph{avoids measures}  if whenever $\mu$ is a measure on $X$ such that $\mu(\{x\})=0$, then $x$ does not belong to the support of $\mu$.
\end{definition}

We can consider points avoiding some specific type of measures. We will be particularly interested in avoiding non-atomic measures so it will be convenient to introduce the
following notation.
	
\begin{definition}
A point is a $\mu$-point if it avoids non-atomic measures.
\end{definition}


A topological space has the \emph{countable chain condition} (is ccc, in short) if it does not contain an uncountable family of pairwise disjoint open sets. By $\mathbb{B}$ we denote the \emph{measure algebra}, i.e. the Boolean algebra
$\mathrm{Bor}(2^\omega)/_{\mathcal{N}}$, where $\mathcal{N}$ denotes the $\sigma$-ideal of Lebesgue measure zero sets. For $\kappa\geq \omega $ by $\mathbb{B}_\kappa$ we mean the Boolean algebra $\mathrm{Bor}(2^\kappa)/_{\lambda_\kappa=0}$, where
$\lambda_\kappa$ is the standard Haar measure on $2^\kappa$.

For $A,B \subseteq \omega$, we write $A\subseteq^\ast B$ if $A\setminus B$ is finite and $A=^\ast B$ if $A\triangle B$ is finite. 

The Stone-\v{C}ech compactification of the non-negative integers $\beta \omega= St(\cc{P}(\omega))$ is the Stone space of the Boolean algebra $\cc{P}(\omega)$. 
By $\omega^{\ast}$ we denote the remainder of Stone-\v{C}ech compactification of $\omega$, i.e. $\omega^{\ast}= St(\cc{P}(\omega)/ fin))$.

If $X\subseteq \omega$, then by $[X]$ we will denote the clopen subset of $\beta\omega$ (or of $\omega^\ast$) induced by $X$, i.e. $[X] = \{x\in \beta\omega\colon X \in x\}$. Similarly, if $v\in 2^{<\omega}$, then by $[v]$ we will denote
the clopen subset of $2^\omega$ induced by $v$: $[v] = \{x\in 2^\omega\colon x\mbox{ extends }v\}$.

\begin{definition}
	A point $p \in X$ is a \emph{$P$-point} if for every countable family of neighbourhoods $\{U_n\colon n\in \omega\}$ of $p$ there exists open set $U$ such that $p \in U \subseteq \bigcap \limits_{n} U_n$.
\end{definition}



\begin{theorem}[W. Rudin, \cite{Rudin}]
Under (CH) there exists a $P$-point in $\bw$.
\end{theorem}

It is also known that there is a model without $P$-points, see \cite{Shelah}.
However, in ZFC there exist points with slightly weaker property.

\begin{definition}
	A point $p \in X$ is a \emph{weak $P$-point} if $p$ is not a limit point of any countable subset of $X\setminus \{p\}$.  \end{definition}

We say that a closed subset $F\subseteq X$ is a \emph{weak $P$-set} if whenever $C$ is a countable set disjoint with $F$, we have
	$\overline{C} \cap F = \emptyset$.

\section{Results about general topological spaces}\label{general-spaces}

We say that a compact space $K$ supports a measure $\mu$ if $\mu(U)>0$ for every nonempty open $U\subseteq K$ (or, equivalently, if $K$ is the support of $\mu$). 
The property of \emph{not} avoiding measures is in a sense the \emph{local} version of the property of supporting a measure. 
E.g. if a space supports a measure, then of course none of its point can avoid measures. 
We will show now that this implication cannot be reversed, i.e. that there is a space which does not support a measure, but which does not contain $\mu$-points, and so \emph{locally} it supports a measure everywhere.

\begin{prop}\label{strictly-positive}
There is a space without strictly positive measure, and without points avoiding measures.
\end{prop}

\begin{proof}
	Let $X_\alpha=[0,1]$ for every $\alpha\in \omega_1$ and let $X= \bigcup_{\alpha \in \omega_1} X_\alpha$ be the disjoint union of those intervals. Let $K = X \cup \{\infty\}$ be the one-point compactification of $X$. 
	
Clearly, $K$ does not satisfy $ccc$ and so there is no measure $\mu$ such that $supp(\mu)=K$.

For every $\alpha \in \omega_1$ let $\lambda_\alpha$ be the Lebesgue measure supported by $X_\alpha$ and let $\lambda_\omega$ be defined by 
	\[ \lambda_\omega = \sum_{n \in \omega}\frac{1}{2^n}\lambda_n. \]
For every point $x_0 \in X$ there is $\alpha\in \omega_1$ such that $x_0 \in X_\alpha$ and so $x_0\in \mathrm{supp}(\lambda_\alpha)$. 
	
	By the definition of the one
	point compactification, for every open $U\subseteq K$, if $\infty \in U$ then $X_\alpha \subseteq U$ for cofinitely many $\alpha \in \omega_1$. Hence $\infty \in \overline{(\bigcup_{n\in \omega}X_n)}$ and so $\infty \in
	\mathrm{supp}(\lambda_\omega)$.
\end{proof}

We proceed with a fact which shows that the notion of avoiding measures is, in a sense, generalization of the P-point property.

\begin{prop}\label{P-point} If $p$ is a P-point, then $p$ avoids measures.
\end{prop}
\begin{proof}
	If $p$ is a P-point and $\mu$ is a measure such that $\mu(\{p\})=0$, then we can find a decreasing sequence $(U_n)$ of open neighborhoods of $p$ such that $\mu(U_n)<1/n$. There is $U\ni p$ such that $U\subseteq U_n$ for every $n$ and so $\mu(U)=0$. Hence $x\notin \mathrm{supp}(\mu)$.
\end{proof}

\begin{remark}
Although we have promised to deal only with Radon measures, when talking about measures on topological spaces, let us make a remark that the above proposition may be no longer true, if $\mu$ is not outer regular. Indeed, consider $X=\omega_1+1$,
	$p=\omega_1$, and let $\mu$ be the Dieudonné's measure (see \cite[411Q]{Fremlin}). Then $p$ is a P-point, $supp(\mu)=\{p\}$ and $\mu(\{p\})=0$. 
\end{remark}

Since the notion of weak P-point also generalized P-pointness, it is natural to ask how the points avoiding measures and weak P-points are related. It turns out that weak P-points can be defined in terms of avoiding measures.

\begin{prop} A point $p$ is a weak P-point if and only if it avoids purely atomic measures.
\end{prop}

\begin{proof} 
	Notice that since every inner regular 0-1 measure can be expressed as a Dirac delta measure, $\mu$ on a topological space $X$ is purely atomic if and only if it is of the form $\mu=\sum_{n\in \omega}a_n \delta_{x_n}$ for some $a_n\geq0$, $x_n\in X$, $\sum a_n = 1$.

Suppose that $p$ is not a weak P-point and so that there is $C=\{c_n \colon n\in \omega\}\subseteq X\setminus \{p\}$ such that $p\in \overline{C}$. Define $\mu=\sum_n 2^{-n}\delta_{c_n}$. Then $\mu(\{p\})=0$, and if $U$ is a neighbourhood of $p$, then there is $n_0$ such
	that $c_{n_0} \in U$. Hence $\mu(U)\geq 2^{-n_0}\delta_{c_{n_0}}(U)=2^{-n_0}>0$ and so $p \in supp(\mu)$.

	Suppose that $p$ does not avoid purely atomic measures, i.e. there is $\mu=\sum_n a_n\delta_{c_n}$ for appropriate $(a_n)$ and $(c_n)$ such that $\mu(\{p\})=0$ and $p\in supp(\mu)$. As $\mu(\{p\})=0$ we have that $p\neq c_n$ for every $n\in
	\omega$. $p\in supp(\mu)$, for every neighbourhood $U$ of $p$ we have $\mu(U)>0$. Hence, there is $n$ such that $\delta_{c_n}(U)>0$ and so $c_n \in U$. Finally, $p\in \overline{C}$, where $C=\{c_n \colon \ n \in \omega\}$.
\end{proof}

If $X$ is a scattered space, then  by \cite{Rudin2} every measure on $X$ is purely atomic. So, if $X$ is scattered then $x\in X$ avoids measures if and only if $x$ is a weak P-point.

\begin{example} Consider $X = \omega_1+1$ as the ordered topological space. It is scattered, so every measure on $X$ is purely atomic. Hence, $\omega_1$ (as a point in $X$) avoids measures.
\end{example}

However, there are both examples of spaces containing weak P-points but no $\mu$-points and vice versa.

First of all, notice the following simple fact.

\begin{prop}\label{sep} If $K$ is separable, and $x\in K$ is not isolated, then $x$ is not a weak P-point.
\end{prop}

Now we will show that the Stone space of the measure algebra contains a weak P-point. We are grateful to Jan van Mill for presenting us the following proof. It is an easy consequence of his theorem but it seems that it was not stated explicitly in
the literature. Of course none of the points in the Stone space of the measure algebra can avoid the Lebesgue measure. So, it does not contain a $\mu$-point.

Recall that if $K$ is the Stone space of the measure algebra $\mathbb{B}$ and $\lambda$ is the standard measure on $K$, then $\lambda$ is inner (and outer) regular with respect to clopen subsets. Also, $\mathbb{B}$ is homogeneous, in the sense that $\{A\in
\mathbb{B}\colon A\subseteq B\}$ is isomorphic to $\mathbb{B}$ for each non-zero $B\in \mathbb{B}$. Consequently, if $L \subseteq K$ is separable, then $\lambda(L)=0$ (as $K$ is not separable). For basic properties of $K$ and $\lambda$, see e.g.
\cite[Section 24]{Halmos}.

\begin{theorem}\label{Janek} There is a weak P-point in the Stone space of the measure algebra $\mathbb{B}$.
\end{theorem}

\begin{proof} Let $K = \mathrm{St}(\mathbb{B})$.

Let $(C_n)$ be a family of pairwise disjoint clopens of $K$ such that $\bigcup_n C_n$ is dense in $K$. Let $X=\bigcup_n C_n$. For $n\in \omega$ let $\mathcal{F}_n$ be the family defined by 
	\[ \mathcal{F}_n = \{F\subseteq C_n\colon F\mbox{ is clopen and } \lambda(C_n \setminus F)< \lambda(C_n)/(n+1)\} \]
and notice that for each $n$ the family $\mathcal{F}_n$ is $n$-linked.

Let 
	\[ \mathcal{F} = \{F \subseteq X \colon F\cap C_n \in \mathcal{F}_n \mbox{ for each }n\}. \]
	Then $\mathcal{F}$ is a filter on the Boolean algebra of clopen subsets of $X$ and by \cite[Theorem 2.5]{vanMill} it can be extended to an ultrafilter $x\in \beta X$ which is a weak P-point in $X^\ast$.

	Since $K$ is extremely disconnected, we have $K = \beta X$ and so $x\in K$. Now, suppose that $N$ is a countable subset of $K$. We want to show that $x\in \overline{N}$ and we may assume that either $N\subseteq X$ or $N\subseteq K\setminus X =
	X^\ast$. Suppose, first, that  $N\subseteq X^\ast$. Then $x\notin
	\overline{N}$ as $x$ is a weak P-point of $X^\ast$. If $N\subseteq X$, then there is $F \in \mathcal{F}$ such that $\overline{N}\cap F
	= \emptyset$. Indeed, $\lambda(\overline N)=0$ (see the remarks above) and so, by regularity of $\lambda$ with respect to clopens, for every $n$ we can pick $F_n\in \mathcal{F}_n$ such that $F_n \cap \overline{N}=\emptyset$. Then $F = \bigcup_n
	F_n \in \mathcal{F}$ and $F \cap \overline{N}=\emptyset$. Hence $x\notin \overline{N}$.  
\end{proof}

\begin{theorem} There is a space $X$ which is separable (and so, there is no weak P-point in $X$) but which contains $\mu$-point (in fact, a dense set of $\mu$-points).
\end{theorem}

\begin{proof} By \cite{Simon_p_set}, see also \cite{Dzam_Pleb} there is a separable space $X$ with a countable, dense, dense-in-itself subset $H\subseteq X$, which is a P-set in X.
Given a non-atomic measure $\mu$, since $H$ is countable, we have $\mu(H)=0$. By the same argument as in the proof of Proposition \ref{P-point} we can find a neighbourhood $U$ of $H$ such that $\mu(U)=0$. Hence every point of $H$ is a $\mu$-point. Since $X$ is separable, there are no weak P-points.
\end{proof}

\begin{remark}\label{betaomega} In fact there is a 'simpler' (or, at least, more classical) example of a space containing a $\mu$-point which is not a weak P-point. Namely, $\beta\omega$ contains a non-isolated point which is a $\mu$-point (and which is not a weak
	P-point because of separability of $\beta\omega$). We will come back to this fact in the next section, armed with additional notions and facts.
\end{remark}

We finish this section with some general results bonding the existence of weak P-points and $\mu$-points.

\begin{lemma} \label{mu_is_not_limit}
A $\mu$-point cannot be a limit point of a countable set of non$-\mu$-points.
\end{lemma}

\begin{proof}
Suppose that a set $A=\{a_n\colon n\in \omega\}$ is countable and none of $a_n$ is a $\mu$-point. Assume that $p\in \overline{A}$. We will show that $p$ does not avoid non-atomic measures.

	Let ${\mu}_n$ be a non-atomic measure such that $a_n \in \mathrm{supp}(\mu_n)$. Define \[ \mu= 
	\sum_{n \in \omega}\frac{1}{2^n}{\mu}_n.\] Since $p\in \overline{A}$, for every open $U$ if $p\in U$, then there is $n_0 \in \omega$ such that $a_{n_0} \in U$. Then $\mu(U)\geq 2^{-n_0}{\mu}_{n_0}(U)>0$ and so $\mu(U)>0$ for every neighbourhood
	$U$ of $p$. Hence $p \in \mathrm{supp}(\mu)$. 
\end{proof}

\begin{prop}
	If a compact space $K$ contains a $\mu$-point, then it contains a point which is $\mu$- and weak P- simultaneously (so, a point which avoids all measures), or a countable, dense-in-itself subspace of $\mu$-points.
\end{prop}
\begin{proof}
	Suppose there is at least one $\mu$-point $p \in X$, and that all $\mu$-points in $X$ are not weak P-points. We inductively construct the subspace with desired property. Let $A_0=\{p\}$ and $p$ is $\mu$-point.  Suppose $A_i=\{a_i^{n}\colon n \in \omega\}$ is constructed and
	consists of $\mu$-points. For every $n$, as $a_i^{n}$ is a $\mu$-point, by the assumption and by Lemma \ref{mu_is_not_limit}, there is a countable set $A_{i+1}^{n}$ of $\mu$-points, such that $a_i^{n} \in \overline{A_{i+1}^{n}}$. Let $A_{i+1}=\bigcup_{n\in
	\omega}A_{i+1}^{n}$. By the construction, $A_i \subseteq \overline{A_{i+1}}$.

	Let $A=\bigcup_{i\in \omega}A_{i}$. Then $A$ is a countable dense-in-itself set of $\mu$-points.
\end{proof}

The last result will use some forcing (but without going into forcing technicalities). Let us first make some general comments. Consider a model $M$ and its forcing extension $M[G]$. Let $\mathbb{A}$ be a Boolean algebra, an element of the model $M$. Then
$\mathbb{A}$ is also an element of $M[G]$. If $x$ is an element of the Stone space of $\mathbb{A}$, then it is still an
element of the Stone space of $\mathbb{A}$ in $M[G]$. Typically the Stone space of $\mathbb{A}$ computed in $M[G]$ will contain more elements and, in general, the Stone space of $\mathbb{A}$ in $M[G]$ may have different topological
properties than the Stone space of $\mathbb{A}$ in $M$. 

\begin{prop} Suppose $\mathbb{A}$ is a Boolean algebra and $K$ is its Stone space. If $x\in K$ does not avoid measures, then there is $\kappa$ such that $\mathbb{B}_\kappa$ forces that $x$ is not a weak P-point in $St(\mathbb{A})$.
\end{prop}

\begin{proof} Let $\mathbb{A}$, $K$ and $x\in K$ be as above. Let $\mu$ be the measure on $K$ such that $x\in F = \mathrm{supp}(\mu)$. By the Stone duality, there is a Boolean algebra $\mathbb{A}'$ such that $St(\mathbb{A}') = F$ and there is a
	Boolean homomorphism $\varphi\colon \mathbb{A} \to \mathbb{A}'$ which is onto. Also, as $x\in K$, $x$ extends (as an ultrafilter) the filter dual to the kernel of $\varphi$. Since $\mathbb{A}'$ supports a measure, by \cite{Kamburelis}, there is $\kappa$ such
	that $\mathbb{B}_\kappa$ forces that the Stone space $K'$ of $\mathbb{A}'$ (computed in the extension) is separable. The mapping $\varphi$ belongs to the extension, so in the extension $x\in K'$. Hence, it is not a weak P-point.
\end{proof}

\section{$\mu$-points versus weak P-points in $\omega^\ast$}\label{omegaast}

In this section we consider the relationship between weak P-points and $\mu$-points in $\omega^\ast$. We begin with a rather easy result: there are points in $\omega^\ast$ which are neither weak P-points nor $\mu$-points. Then we will show how to
use Kunen's $\mathfrak{c}$-OK-points to avoid all measures (so to produce points which are both weak P-points and $\mu$-points). Then we consider the following natural question: if a point $x\in \omega^\ast$ avoids non-atomic measures, should it avoid also purely-atomic ones (and so should it be a weak $P$-point) and vice versa?

\begin{prop} There is an ultrafilter which is neither $\mu$-point nor a weak P-point.
\end{prop}
\begin{proof} By \cite{Pelcz_Semad} there is a non-atomic measure $\mu$ on $\omega^\ast$. Let $C$ be a countable discrete subset of the support of $\mu$ and let $x\in \overline{C}\setminus C$. Then $x$ is not a weak P-point (which is witnessed by
	$C$) and $x$ does not avoid the measure $\mu$.
\end{proof}

The easiest example of a point of $\omega^\ast$ which avoids all measures (and so it is both a $\mu$-point and a weak P-point) is a P-point (see Proposition \ref{P-point}). Consistently, there are no P-points but nevertheless the points avoiding
measures do exist in ZFC.
	
\begin{theorem}\label{Avoid-measures} There is a point which is a weak P-point and $\mu$-point simultaneously (and so it avoids all measures).
\end{theorem}

To prove the above we will use so called $\mathfrak{c}$-OK point introduced by Kunen in \cite{Kunen_weak_p}. Kunen proved that $\mathfrak{c}$-OK-points cannot belong to a closure of any ccc set (not containing the point itself). The above theorem is an immediate
consequence of this fact. We will enclose Kunen's proof for the readers convenience. 

\begin{definition}
	A set $S\subseteq X$ is a \emph{$\kappa$-OK set} if for every countable collection $\{U_n\colon n\in \omega\}$ of neighbourhoods of $S$ there is a family of $\kappa$ many neighbourhoods $\{V_{\alpha}\colon \alpha<\kappa\}$ (called a refinement system of
	$\{U_n\colon n\in \omega\}$) such that for every finite $n$ and $\alpha_{0}\dots \alpha_{n-1}$ we have $V_{\alpha_0}\cap \dots \cap V_{\alpha_{n-1}}\subseteq U_n$. We say that a point $p$ is a \emph{$\kappa$-OK point} if $\{p\}$ is a $\kappa$-OK
	set.
\end{definition}

\begin{theorem}[Kunen]
	There are $\mathfrak{c}$-OK-points in $\omega^\ast$.
\end{theorem}

The crucial fact connecting $\kappa$-OK-points with the points avoiding measures is the following:

\begin{theorem}[Kunen]
Suppose $X$ is a $T_3$ ccc space. Let $p \in X$ be a non isolated point. Then $p$ is not $\omega_{1}$-OK.
\end{theorem}

\begin{proof}
	First, notice that if $p$ is as above, then $p$ is not a P-point. Otherwise we could find a family $\{ U_\alpha\colon  \alpha<\omega_1 \}$ of neighbourhoods of $p$, such that $cl(U_\beta)$ is a proper subset of $U_\alpha$ whenever
	$\alpha<\beta$. But then $\{U_\alpha \setminus \overline{U_{\alpha+1}}: \alpha< \omega_1\}$ would contradict the ccc property.

	Now, let $\{ U_n: \ n \in \omega \}$ be a family of open neighborhoods of $p$ such that \[ p \not \in \mathrm{Int}\left(\bigcap \limits_{n}\overline{U_n}\right).\] Let $\{ V_\alpha: \alpha< \omega_1 \}$ be a refinement system for $\{ U_n: n \in
	\omega \}$. For each $\alpha<\omega_1$, there is $n\in \omega$ such that $V_\alpha \setminus \overline{U_n}\neq \emptyset$, so fix $n \in \omega$ so that $B= \{\alpha<\omega_1: \ V_\alpha \setminus \overline{U_n}\neq \emptyset\}$ is uncountable.
	For $\alpha \in B$, let $W_\alpha= V_\alpha \setminus \overline{U_n}$. Then $W_\alpha$ are non-empty, but any $n$ of them have empty intersection. By the Lemma \ref{Lemma}, it contradicts ccc.
\end{proof}

\begin{lemma}\label{Lemma}
	Suppose $\cc{A}$ is a family of open subsets of a space $X$, $|\mathcal{A}| = \kappa$  and the intersection of any $n$ sets from the family is empty. Then there are $\kappa$-many pairwise disjoint open subsets of $X$.
\end{lemma}
\begin{proof}
Let $\cc{A}=\{A_\alpha\colon \ \alpha<\kappa\}$ and define \[ \cc{A}_l=\{A_0 \cap \dots \cap A_{l-1}\neq \emptyset \colon \ A_j \in \cc{A}, \ j<l\}.\] 
	Find $m<n$ such that $|\cc{A}_m|=\kappa$, and $|\cc{A}_{m+1}|=\xi< \kappa$. Let \[ J=\{\alpha<\kappa:  \ \exists \alpha_1 \dots \alpha_{m} \ (A_{\alpha} \cap A_{\alpha_1} \cap \dots \cap A_{\alpha_m})\in \cc{A}_{m+1} \} \]  and let $I= \kappa
	\setminus J$. Then $|J|\leq |\xi \cdot (m+1)|=\xi$. Let \[ \cc{C}=\{U\in\cc{A}_m: \ \exists \beta \in I \ \exists \alpha_1, \dots \alpha_{m-1} \ U=(A_\beta \cap A_{\alpha_1}\cap \dots \cap
	A_{\alpha_{m-1}})\}.\]  Notice that since $|J^{m}|=\xi$, we have $|\cc{C}|=\kappa$. As $J \cap I= \emptyset$, $\cc{C}$ is a family of pairwise disjoint open sets.
\end{proof}

So far, we have the following picture in $\omega^\ast$.

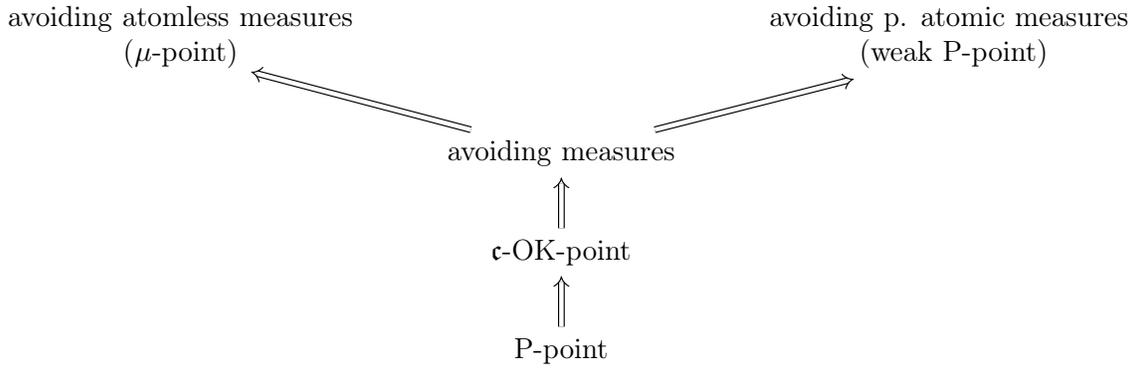
\begin{center}
	\begin{tikzcd}[arrows=Rightarrow]
		\mbox{avoiding atomless measures}  &   &  \mbox{avoiding p. atomic measures } \\[-25pt]
			\mbox{(}\mu\mbox{-point)}  &  &\mbox{(weak P-point)} \\
		&\arrow[lu]		\mbox{avoiding measures} \arrow[ru]  &  \\
		& \arrow[u] \mathfrak{c}\mbox{-OK-point}\\
&\arrow[u]		\mbox{ P-point }  &  \\
\end{tikzcd}
\end{center}
\label{diagram}

We will now study the relation between weak P-points and $\mu$-points.

\begin{theorem} There is a weak $P$-point which is not $\mu$-point.
\end{theorem}
\begin{proof} By \cite{Simon_ind} if $\cc{B}$ is a complete Boolean algebra of cardinality $\mathfrak{c}$, then $St(\cc{B})$ can be
	embedded as a $\mathfrak{c}$-OK subset of $\omega^{\ast}$. Therefore, we can embed the Stone space $Y$ of the measure algebra into $\omega^{\ast}$ as an $\mathfrak{c}$-OK set.
	By Theorem \ref{Janek} there is a weak P-point in the Stone space of the measure algebra. Then $Y$ contains a weak P-point $p$ (relatively to $Y$). Hence, as $Y$ is a weak P-set, $p$ is a weak P-point in $\omega^{\ast}$. There is the canonical
	(Lebesgue)
	non-atomic measure $\lambda$ on $Y$ with $supp(\lambda)= \overline{Y}$. Hence $p\in supp(\lambda)$ and so it is not a $\mu$-point.
\end{proof}

Now, we will come back to Remark \ref{betaomega}. Let $x$ be a $\mathfrak{c}$-OK point in $\beta\omega$. Then, $x$ is a $\mu$-point, but $x$ is not a weak P-point (because of Proposition \ref{sep}). We will use this result to show that there is a
$\mu$-point which is not a weak $P$-point in $\omega^\ast$.

\begin{theorem}
There is $\mu$-point in $\omega^\ast$ which is not a weak $P$-point. 
\end{theorem}
\begin{proof}
	Since $\mathcal{P}(\omega)$ is a complete Boolean algebra, by \cite{Simon_ind}, $\beta\omega$ can be embedded as a $\mathfrak{c}$-OK subset of $\omega^\ast$. Denote this copy by $K$. 
	Let $x\in K$ be a $\mathfrak{c}$-OK point (relatively to $K$). The point $x$ is not a weak P-point (as it is in the closure of the copy of $\omega$ in $K$). We will show that it avoids non-atomic measures.

 Let $\mu$ be a non-atomic
	measure and let $L$ be its support. Then $L\setminus K$ is ccc. Therefore, $\overline{L \setminus K} \cap K = \emptyset$ as $K$ is a $\mathfrak{c}$-OK set. 
	So, we may assume without loss of generality, that $\mu$ is supported by $K$ (considering $\mu_{|K}$ instead of $\mu$, if needed). But then $x$ is $\mathfrak{c}$-OK point in $K$ and so $x \notin \mathrm{supp}(\mu)$. 
\end{proof}



\section{Ultrafilters avoiding measures and measure zero ultrafilters}\label{baum}



The space $\omega^\ast$ is very far from supporting a strictly positive measures (e.g. as it contains an antichain of open sets of size $\mathfrak{c}$). However, as we have learned in Proposition \ref{strictly-positive}, this is not enough to deduce that it contains a $\mu$-point. 
We may try to construct $\mu$-point in $\omega^\ast$ by transfinite induction but we immediately face the following problem: there are too many non-atomic measures to 'kill' them inductively one by one. 

Fortunately, we are able to 'kill' many measures at the same time and in fact we will not 'kill' measures but rather trees of subsets of $\omega$. With every non-atomic measure we can associate a tree $T$ witnessing its non-atomicity. Adding a set appropriate for
this tree will ensure that the resulting ultrafilter will avoid all the measures for which $T$ is a non-atomicity witness.

We will present the details of this idea. Then, we show its connection to measure zero ultrafilters (in the sense of Baumgartner) and the constraints of this method.

So, we will start with several definitions concerning trees. 

A family $\mathcal{T}$ is a \emph{tree of subsets of} $\omega$ if $\mathcal{T} = \{T_s\colon s\in 2^{<\omega}, \ T_s\subseteq \omega\}$ where
\begin{itemize}
	\item if $s,t \in 2^{<\omega}$, and $s$ extends $t$, then, $T_s\subseteq^{\ast} T_t$,
	\item $\bigcup\{T_s \colon s \in 2^k\} =^{\ast} \omega$ and $T_s \cap T_t =^{\ast} \emptyset$ for each $s, t \in 2^k$ for every level $k\in \omega$ (notice that we allow $T_s$ to be empty).
\end{itemize}

A tree $T$ is called \emph{a full tree} if for every $s \in 2^ {<\omega}$ the set $T_s$ is infinite. For a set $X\subseteq \omega$ and a tree $\mathcal{T}$ of subsets of $\omega$ let \[ T_X = \{x \in 2^\omega\colon \forall k \ |T_{x|k} \cap X| = \omega \}. \]

Fix an ideal $\mathcal{I}$ on $2^\omega$. We say that a set $X$ \emph{$\mathcal{I}$-diagonalizes} a tree $\mathcal{T}$ (of subsets of $\omega$) if $T_X \in \mathcal{I}$. An ultrafilter $\mathcal{I}$-diagonalizes a tree $\mathcal{T}$ if it
contains an element $\mathcal{I}$-diagonalizing $\mathcal{T}$. Finally, we say that an ultrafilter \emph{$\mathcal{I}$-diagonalizes the trees} if it $\mathcal{I}$-diagonalize every tree of subsets of $\omega$.

\begin{remark} \label{duality} The property of $\mathcal{I}$-diagonalizing trees can be stated in the topological terms, using the Stone duality. Namely, an ultrafilter $\mathcal{U}$ $\mathcal{I}$-diagonalize (full) trees if and only if for each continuous (surjective) function $f\colon
	\omega^\ast \to 2^\omega$, there is a clopen $C\subseteq \omega^\ast$ such that $f[C] \in \mathcal{I}$.
\end{remark}

\begin{example} Suppose $\mathcal{U}$ is a P-point in $\omega^\ast$. Let $Fin$ denote the ideal of finite subsets of $\omega$. Then $\mathcal{U}$ $Fin$-diagonalizes every tree of subsets of $\omega$. Indeed, every tree of subsets of $\omega$ has a
	branch of elements of $\mathcal{U}$. Let $X \in \mathcal{U}$ be a pseudo-intersection of this branch, promised by P-pointness. Then $|T_X| = 1$. 
\end{example}

So, the property of $\mathcal{I}$-diagonalizing trees, can be seen as a generalization of the P-point property. We will first prove how the ultrafilters which $\mathcal{N}$-diagonalize trees of subsets of $\omega$ can be used to
construct $\mu$-points. Then, we will see that the notion of ultrafilters $\mathcal{I}$-diagonalizing trees is in fact the well-known notion of $\mathcal{I}$-ultrafilters in disguise.

	By \emph{the Cantor algebra} $\mathbb{C}$ we mean the unique, up to isomorphism, countable atomless Boolean algebra. We will identify $\mathbb{C}$ with the $Clop(2^\omega)$, the Boolean algebra of clopen subsets of $2^\omega$. Notice that $\mathrm{St}(\mathbb{C}) =
	2^\omega$.

\begin{lemma}
\label{partition}
	Let $\mu$ be a measure on the Cantor algebra $\mathbb{C}$. For every $k \in \omega$ and for each $\varepsilon>0$, there is a partition $\{S_1, \dots , S_k \}$ of unity of $\mathbb{C}$, such that \[ \forall i\leq k \ \ \ |\mu(S_i)-\frac{1}{k}|<\varepsilon.\]
\end{lemma}
\begin{proof}
	Choose any $\delta>0$ such that \[ \delta < \min(\frac{1}{k}, \frac{\varepsilon}{k-1}).\] By non-atomicity of $\mu$, pick any partition $\{P_1, \dots P_r\}$ of unity of $\mathbb{C}$ into sets of measure smaller than $\delta$. Let $a_i=\mu(P_i)$ for
	$i\leq r$. We are going to construct $S_i$'s for $i<k$ by induction, constructing simultaneously a sequence $(n_i)$ of integers. Let $n_1=1$. For $i<k$ let 
	\[ n_{i+1}=\max\{m \colon \sum_{j=n_i}^{m}a_i \leq \frac{1}{k} \}+1\] and let  \[ S_i=\bigcup \limits_{j\in [n_i, n_{i+1})} P_j.\]
	By maximality of $n_i$'s we have  $ 0<1/k-\mu(S_i)\leq \delta < \varepsilon$ for each $i<k$. Define $S_k=\omega \setminus \bigcup_{i<k}S_i$. Then
\begin{align*}
&\mu(S_k)= k\cdot\frac{1}{k}- \sum_{i=1}^{i<k}\mu(S_i)=
\frac{1}{k}+\sum_{i=1}^{i<k}\left(\frac{1}{k}-\mu(S_i)\right)\leq \frac{1}{k}+(k-1)\delta \leq
\frac{1}{k}+\varepsilon,
\end{align*}
	and so \[ |\mu(S_k)-1/k| < \varepsilon. \]
\end{proof}

In what follows by $\lambda$ we will mean the Lebesgue measure on $2^\omega$.

\begin{theorem}\label{non-atomic-tree}
	For every non-atomic measure $\mu$ on $2^\omega$, there exists a surjective function $f\colon 2^\omega \to 2^\omega$, such that whenever $S\subseteq 2^\omega$ is closed and $\lambda(S)=0$, then $\mu(f^{-1}(S))=0$. 
\end{theorem}
\begin{proof}
	For every sequence $\{\varepsilon_i\colon \ \varepsilon_i>0, \  i \in \omega\}$, it is possible to construct an embedding $i\colon \mathbb{C} \to \mathbb{C}$, such that for every level $k$ and every node $v \in 2^k$ we have \[\max(\mu(i(v^\frown
	0)),\mu(i(v^\frown 1)))\leq
	\frac{1}{2}\mu(i(v))+\varepsilon_{k+1}.\]
	We will construct $i$ by induction.  Suppose that $i(u)$ is defined for all nodes $u \in 2^{\leq k}$. For $v \in 2^k$ consider the Boolean algebra $\mathbb{C}\restriction i(v)$ and the restriction of $\mu$ respectively. Since the restriction of
	a non-atomic measure on an element of positive measure is non-atomic, using Lemma \ref{partition} we can divide $i(v)$ into 2 sets with the 'precision' \[ \varepsilon=\frac{\varepsilon_{k+1}}{\mu(i(v))} \] to define $i(v^\frown 0),i(v^\frown 1)$
	as desired. Continue for all nodes $v\in 2^k$ and all the levels $k$. The function $i\colon 2^{<\omega}\to \mathbb{C}$ can be extended to the desired embedding.

Now, let $\varepsilon_k = 2^{-2k}$ for every $k\in \omega$.
	Let  $f_k, g_k \colon \mathbb{R} \to \mathbb{R}$ be defined by $f_k(x)= \frac{1}{2}x+ \varepsilon_k$, and $g_k=(f_k \circ \dots \circ f_1)$. Then \[ g_k(1)= 2^{-k} + \varepsilon_1 2^{-k+1} + \dots + \varepsilon_k = 2^{-k} + 2^{-k}\sum_{i=1}^k 2^i \varepsilon_i =2^{-k}+2^{-k}\sum_{i=1}^{k}2^{-i}<2 \cdot
	2^{-k}.\]  Notice, that if $i$ is constructed as above, then for every $v\in 2^k$ we have $\mu(i(v))\leq g_k(1) < 2\cdot \lambda([v])$ (where $[v]$ is the clopen subset of $2^\omega$ generated by $v$). Suppose that $S$ is a closed subspace of
	$2^\omega$ such that $\lambda(S)=0$. Since $S$ is compact, and $\lambda$ is outer regular with respect to open sets, given $\delta>0$ we can find a clopen $U$, such that $S\subseteq U$,
	$\lambda(U)\leq \frac{\delta}{2}$, and without loss of generality $U= \bigcup_{j\leq n} [v_j]$ for some $n, \ k$ and $v_j \in 2^k$. 
	
	Let $f\colon 2^\omega \to 2^\omega$ be the continuous map dual to the embedding $i$, i.e. $f(y) = \bigcap \{[i(C)]\colon y \in [C]\}$. 

	Then 
 \[ \mu({f}^{-1}(S))\leq \mu({f}^{-1}[U])= \mu(i(U))= \sum_{j \leq n} \mu(i(s_j)) \leq 2 \cdot \sum_{j \leq n} \nu(s_j)= 2 \cdot \nu(U)< \delta\]
	As $\delta$ was arbitrary, $f$ is as desired.
\end{proof}

\begin{corr}
	For every non-atomic measure $\mu$ on $\beta \omega$, there is a tree $\mathcal{T}$ such that for every $U \subseteq \omega$ satisfying $T_U\in \mathcal{N}$ we have $\mu[U]=0$. Consequently, if an ultrafilter $\mathcal{N}$-diagonalizes full trees, then it avoids
	non-atomic measures.
\end{corr}

\begin{proof}
	By non-atomicity of $\mu$ we can construct a surjective function $f \colon \beta\omega \to 2^\omega$ such that the image measure $\nu(\cdot)= \mu(g^{-1}(\cdot))$ is non-atomic on $2^\omega$.

	
	By Theorem \ref{non-atomic-tree} we can find a surjective function $g\colon 2^\omega \to 2^ \omega$ such that for every closed $S \subseteq 2^\omega$ if $\lambda(S)=0$, then $\nu(g^{-1}(S))=0$. 
	Let $\mathcal{T}$ be the tree of subsets of $\omega$ induced in a natural way by the tree $\{(g \circ f)^{-1}[v]\colon v\in 2^{<\omega}\}$ (of subsets of $\beta\omega$). Then, by surjectivity of the functions $f$ and $g$, the tree $\mathcal{T}$ is
	full. Also,  $T_U = (g \circ f)[ U ]$ for each $U
	\subseteq \omega$ (here we treat $U$ once as a subset of $\omega$ and once as a clopen subset of $\beta\omega$).

	Therefore, given that $\lambda(g\circ f[U])=0$ for some $U \subseteq \omega$, we have \[ \mu([U])\leq \mu(f^{-1} ( f [U]))= \nu(f[U])\leq \nu( g^{-1} (g (f[U])))=0.\]

\end{proof}

Using the above fact one can construct in a standard way, by a transfinite induction, a $\mu$-point which is not a P-point (at least under $CH$). It is enough to make sure to $\mathcal{N}$-diagonalize every tree of subsets of $\omega$. Instead of performing the construction, we will
show that ultrafilters $\mathcal{I}$-diagonalizing trees are in fact well known $\mathcal{I}$-ultrafilters, in the sense of Baumgartner. If $\mathcal{I}$ is the ideal of null sets, such ultrafilters are called measure zero ultrafilters.

\begin{definition}
Let $\mathcal{I}$ be an ideal on $2^\omega$ (containing singletons). We say than an ultrafilter $\mathcal{U}$ is an $\mathcal{I}$-ultrafilter if for every function $f\colon \omega \to 2^\omega$ there is $X \in \mathcal{U}$ such that $\overline{f[X]} \in \mathcal{I}$. 
\end{definition}

\begin{prop} Suppose $\mathcal{I}$ is a $\sigma$-ideal on $2^\omega$. An ultrafilter $\mathcal{U}$ $\mathcal{I}$-diagonalizes trees  if and only if it is an $\mathcal{I}$-ultrafilter. \end{prop}

\begin{proof}
	($\implies$) Consider a function $f\colon \omega \to 2^\omega$. By Remark \ref{duality} we can find $V\subseteq \omega$, such that for  $C=[V] \cap \omega^{\ast}$ we have that $\beta f (C) \in \cc{I}$ . Then $\overline{f(V)}= \beta f([V])=\beta
	f(C) \cup \beta f (V)$. Since $V$ is countable, we are done.

	($\impliedby$) Given $V\subseteq \omega$, such that $\overline{f(V)}\in \cc{I}$, for $C=[V] \cap \omega^{\ast}$ we have $\overline{f(V)}= \beta f([V])=\beta f(C) \cup \beta f (V)\supseteq \beta f(C)$.
\end{proof}

\begin{corr} The ultrafilters $\mathcal{N}$-diagonalizing trees
	are exactly the measure zero ultrafilters and so every measure zero ultrafilter is a $\mu$-point.
\end{corr}

The measure zero ultrafilters were studied by Baumgartner in \cite{Baumgartner} and by Brendle in \cite{Brendle}. Let us mention two relevant results which gives examples of $\mu$-points which are not P-points.

\begin{theorem}(Theorem 1.4 in \cite{Baumgartner}) Under Martin's Axiom for $\sigma$-centered posets, there is a measure zero ultrafilter which is not a P-point. 
\end{theorem}

\begin{theorem}(\cite{Brendle}) It is consistent that there are measure zero ultrafilters but there are no P-points.
\end{theorem}

However, one cannot hope to prove the existence of $\mu$-points in ZFC using measure zero ultrafilters:

\begin{theorem}(\cite{Shelah2}) Consistently, there are no measure zero ultrafilters.
\end{theorem}
\begin{proof} 
	Every measure zero ultrafilter is a nowhere dense ultrafilter (i.e. $\mathcal{I}$-ultrafilter for $\mathcal{I}$ being the ideal of nowhere dense sets), see \cite[Theorem 1.2]{Baumgartner}. Shelah proved (see \cite{Shelah2}) that consistently there
	are no nowhere dense ultrafilters.
\end{proof}

The last theorem shows that the property of being a measure zero ultrafilter is significantly stronger than the property of being a $\mu$-point. Anyway, perhaps the above approach is still worth studying: to prove that an ultrafilter is a $\mu$-point
one does not need to diagonalize all the trees. E.g.  it is natural to care only about full trees. However, as the following proposition shows this does not give us more general notion. 

\begin{prop} An ultrafilter $\mathcal{U}$ $\mathcal{\cc{N}}$-diagonalizes full trees if and only if it $\mathcal{\cc{N}}$-diagonalizes trees.\end{prop}

\begin{proof}
	Of course we need to prove only ($\implies$).

	Instead of trees we will deal with continuous functions as in Remark \ref{duality}. Given a continuous function $f\colon \omega^\ast \to 2^\omega$ let $Y=f(\omega^\ast)$. We may assume that $\lambda(Y)=m>0$ and that $Y$ does not have isolated
	points (otherwise it is easy to find a clopen set as in Remark \ref{duality}). Define a measure $\mu$ on $Y$ by \[ \mu(S)=\frac{\lambda(S)}{\lambda(Y)}.\] 
	The measure $\mu$ is non-atomic. Indeed, for $\varepsilon>0$ find $n$ such that $2^{-n}<\varepsilon \cdot m$. Then for a clopen partition $\{[v]\colon v \in 2^n\}$ we have \[\mu([v] \cap Y)=\frac{\lambda([v] \cap Y)}{\lambda(Y)}\leq \frac{\varepsilon \cdot m}{m}=\varepsilon\]
	
	By Theorem \ref{non-atomic-tree} we get a surjective function $g\colon Y \to 2^\omega$ such that for every closed $S\subseteq 2^\omega$, if $\lambda(S)=0$, then $\mu(g^{-1}(S))=0$. Then $g \circ f \colon \omega^\ast \to 2^\omega$ is surjective,
	and using Remark \ref{duality}, we can find a clopen set $C\subseteq \omega^\ast$  such that  $\mathcal{U}\in C$ and $\lambda((g\circ f)(C))=0$. 
	Hence
	\[ \lambda(f(C)) \leq  \mu(f(C)) \leq \mu( g^{-1} (g \circ f(C)) ) = 0. \] 
\end{proof}

Perhaps diagonalizing even more particular family of trees would be enough. However, we do not have any candidate for such family.

\section{$\mu$-points and Q-points}\label{q-points}

Although it is not related to the main subject of this article, perhaps it is worth to mention the relationship between $\mu$-points and $Q$-points. 

We will show that the property of being a $Q$-point does not imply (at least consistently) the property of avoiding non-atomic measures. This may be somehow surprising as elements of a $Q$-point can be arbitrarily 'rare' as subsets of
$\omega$. So, a $Q$-point cannot extend any density filter (and the standard non-atomic measures on $\omega^\ast$ are exactly extensions of density filters, see e.g. \cite{Blass2}).

Recall that for an ideal $\mathcal{I}$ on the real line, the covering number of $\mathcal{I}$, $\mathrm{cov}(\mathcal{I})$, is defined as the smallest cardinality of a family $\mathcal{A}\subseteq \mathcal{I}$ which covers $\mathbb{R}$. By
$\mathcal{M}$ we denote the $\sigma$-ideal of the meager subsets of $\mathbb{R}$. The dominating number, $\mathfrak{d}$, is defined as the smallest cardinality of the family $\mathcal{A}\subseteq \omega^\omega$ which is dominating, i.e. such that for every
$f\in \omega^\omega$ there is $g\in \mathcal{A}$ such that $g(n)\geq f(n)$ for every $n$.

\begin{theorem}
Consistently, there is a non-atomic measure $\mu$ on $\omega$ and a $Q$-point $x$ such that $x \in \mathrm{supp}(\mu)$.
\end{theorem}

\begin{proof} Consider the classical random model, i.e. the model obtained by adding $\omega_2$ random reals to the model of $\mathsf{CH}$.

     By \cite[Theorem 5.6]{PbnSobota} in this model there is a filter $\mathcal{F}$ generated by $\omega_1$ sets such that the set \[ F = \{x\in \omega^\ast\colon x \mbox{ extends } \mathcal{F} \} \]
     supports a measure.

     By \cite[Theorem 3]{Canjar} if $\mathfrak{d}=\mathrm{cov}(\mathcal{M})$, then every filter generated by less than $\mathfrak{c}$ many sets can be extended to a $Q$-point. In the classical random model both $\mathfrak{d}=\omega_1$ and
     $\mathrm{cov}(\mathcal{M}) = \omega_1$ (see e.g. \cite[Table 4]{Blass}) and so we can extend the filter $\mathcal{F}$ to a $Q$-point $x$. Then $x\in \mathrm{supp}(\mu)$.
\end{proof}

\section{Open problems}

We can play with the definition of measure avoiding points in many ways. E.g. we can consider some particular properties of measures and ask about the existence of points avoiding measures with this property. For example, we may ask about points avoiding measures of countable Maharam type or (in the context of $\omega^\ast$) points
avoiding P-measures (also called (AP)-measures, see \cite{Blass2} or \cite{PbnSobota}).

In fact we may put this kind of questions in a much more general framework. Let $X$ be a (compact) topological space and let $(\varphi)$ be a topological property. We say that a point $x\in X$ \emph{avoids spaces with} $(\varphi)$ if for each space
$Y\subseteq X$ enjoying $(\varphi)$ and such that $x\notin Y$
we have $x\notin \overline{Y}$.

In this sense the points avoiding measures are exactly the points avoiding supports of measures and weak P-points are exactly the points avoiding separable spaces. One can think that the property of \emph{not} avoiding spaces with $(\varphi)$ is some
kind of local version of the
property $(\varphi)$. Or, that a point avoiding spaces with $(\varphi)$ is a \emph{single} witness for non $(\varphi)$. E.g. if a space has a weak P-point, then it is not only non-separable but there is a concrete point witnessing its
non-separability.

This approach motivates questions of the following form.

\begin{problem}\label{witness} Let $(\varphi)$ be a topological property. Is it true that whenever a (compact) space $X$ does not have $(\varphi)$, then it contains a point avoiding spaces with $(\varphi)$?
\end{problem}

Very probably the set of properties for which this is an interesting problem is very limited. In this article we have considered points avoiding spaces possessing some chain condition (separability, being a support of a measure, ccc). Perhaps chain
conditions are the first properties to look at in the context of Problem \ref{witness}. 

Distinguishing different chain conditions was a source of many important problems and results in topology (see e.g. \cite{Todorcevic}). It is natural to ask about distinguishing the 'local' versions of chain conditions. E.g. Theorem
\ref{Janek} distinguishes local separability from local support of a measure. 

\begin{problem}\label{chain} Is there a compact space $X$ and $x\in X$ which avoids measures but not ccc spaces (and, particularly, is not a $\mathfrak{c}$-OK-point)? What about other chain conditions?
\end{problem}

Notice that the standard (consistent) example of a small compact space which is ccc but is not separable is the Stone space of the Boolean algebra generated by the Suslin tree. But
it is not hard to see that this space does not contain a point avoiding separable
spaces (and so it is a non-separable space without a point 'witnessing' its non-separability). So, the examples distinguishing some chain conditions do not automatically distinguish their local versions. (However, each ccc compact F-space of weight bigger than $\mathfrak{c}$ necessarily contains a weak P-point, see \cite{Dow}, so we
can distinguish local version of separability and ccc in ZFC).

Problem \ref{chain} is connected to the main problem of this article. 

\begin{problem} Is there a 'simple' ZFC construction of a $\mu$-point?
\end{problem}

In Section \ref{baum} we tried to attack this problem by 'killing' trees witnessing non-atomicity of measures. This approach led us to measure zero ultrafilters, which consistently may not exist, so perhaps this is a dead-end.

We still hope that there might be some not so complicated measure theoretic reason that $\omega^\ast$ contains $\mu$-points. Mathematicians tried to show non-homogeneity of $\omega^\ast$ using combinatorial and topological arguments. We found it
tempting to try to find an argument of measure theoretic origin.

\end{document}